\documentclass[11pt,twoside, final]{amsart}
\copyrightinfo{0}{Iranian Mathematical Society}
\pagespan{1}{\pageref*{LastPage}}
\usepackage{etoolbox,lastpage}
\commby{}
\date{\scriptsize   Received: , Accepted: .}
\usepackage{amsmath,amsthm,amscd,amsfonts,amssymb,enumerate}
\usepackage{graphicx}		
\usepackage{color}
\usepackage[colorlinks]{hyperref}
 
\newtheorem{theorem}{Theorem}[section]

\newtheorem{lemma}[theorem]{Lemma}
\newtheorem{corollary}[theorem]{Corollary}
\theoremstyle{definition}
\newtheorem{definition}[theorem]{Definition}
\newtheorem{example}[theorem]{Example}
\theoremstyle{remark}
\newtheorem{examples}{Examples}[section]
\newtheorem{remark}[theorem]{Remark}
\numberwithin{equation}{section}
 
\begin{document}

 
\title[GROUP REPRESENTATIONS ON REFLEXSIVE SPACES]{GROUP REPRESENTATION ON REFLEXSIVE SPACES} 
 
\author[B.Khodsiani]{Bahram Khodsiani}
\address[B.Khodsiani]{Department of Mathematics, University of Isfahan, Isfahan.}
\email{${\texttt {b}}_{-}$khodsiani@sci.ui.ac.ir}

\author[A.Rejali]{Ali Rejali $^*$}
\address[A. Rejali]{Department of Mathematics, University of Isfahan, Isfahan.}
\email{rejali@sci.ui.ac.ir}

  \thanks{$^*$Corresponding author}
%
 
 \maketitle
%

\begin{abstract}
Algebras which admit representations on reflexive Banach spaces seem to be a good generalisation of Arens regular Banach algebras, and behave in a similar way to $C^*$-algebras and  Von Neumann  algebras. Such algebras  are called weakly almost periodic Banach algebras (or in abbreviated form  $\mathcal{WAP}$-algebras). In this paper, for weighted group convoltion measure algebra  we construct a representation on reflexsive space.\\
\textbf{Keywords:}  WAP-algebra, dual Banach
algebra, Arens regularity, weak almost
periodicity\\
\textbf{MSC(2010):} 43A10, 43A20, 46H15, 46H25
\end{abstract}
 \section{Introduction and Preliminaries}
The dual $A^*$ of a Banach algebra $A$ can be turned into a Banach $A-$module in a natural way, by setting
$$\langle  f\cdot a, b\rangle=\langle f ,ab\rangle\\\  {\rm and}\\\    \langle a\cdot f, b\rangle=\langle f, ba\rangle $$ for all $ a,b\in A, f\in A^*$.
A {\it dual Banach algebra} is a Banach algebra $A$ such that $A=(A_*)^*$, as a Banach space,
for some Banach space $A_*$, and such that $A_*$ is a closed $A-$submodule of $A^*;$ or equivalently,
the multiplication on $A$ is separately weak*-continuous.
A functional $f\in A^*$ is said to be {\it weakly almost periodic} if $\{f\cdot a: \|a\|\leq 1\}$
 is relatively weakly  compact in $A^*$. We denote by $\mathcal{WAP}(A)$ the set of all weakly
almost periodic elements   of $A^*.$ It is easy to verify that, $\mathcal{WAP}(A)$ is a (norm) closed subspace of $A^*$. As pointed out by Pym\cite{Pym}, $\lambda\in A^*$ is weakly almost periodic if and only if $\lim_m\lim_n\langle a_mb_n,\lambda\rangle=\lim_n\lim_m\langle a_mb_n,\lambda\rangle$ whenever $(a_m)$ and $(b_n)$ are sequences in unit ball of $A$ and both repeated limits exist. For more about weakly almost periodic functionals, see \cite{DL}.
It is known that the multiplication of a Banach algebra $A$ has    two natural but, in general, different extensions (called Arens products) to the second dual $A^{**}$ each turning $A^{**}$  into a Banach algebra. When these extensions are equal, $A$ is said to be (Arens) regular. It can be verified that $A$ is Arens regular if and only if $\mathcal{WAP}(A)=A^*$. Further  information  for the Arens regularity of Banach algebras can be found in \cite{D, DL}.
If $A$ and $B$ are  Banach  algebras, the linear operator $\phi:A\rightarrow B$ is said to be  bounded below if $\inf \{||\phi(a)||:||a||\geq 1\}>0$.
$\mathcal{WAP}$-algebras, as a generalization of the Arens regular algebras, has been introduced and extensively studied in \cite{Da2}. Indeed, a Banach algebra $A$ for which the natural embedding of $A$ into $\mathcal{WAP}(A)^*$ is bounded below, is called a  {\it $\mathcal{WAP}$-algebra}. It has also known that $A$ is a $\mathcal{WAP}$-algebra if and only if it admits an isomorphic representation on a reflexive Banach space. If $A$ is a $\mathcal{WAP}$-algebra, then $\mathcal{WAP}(A)$ separates the points of $A$ and  so $\mathcal{WAP}(A)$ is $\omega^*$-dense in $A^*$.
It can be readily verified that every dual Banach algebra, and every Arens regular Banach algebra, is a $\mathcal{WAP}$-algebra for comparison see \cite{Bkh}. Moreover group algebras are also always $\mathcal{WAP}$-algebras, however,  they are neither dual Banach algebras, nor Arens regular in the case where the underlying group is not discrete, see \cite[Corollary3.7]{B} and \cite{y1,Bkh}.

The paper is organized as follows.  In section 2, we construct a representation on reflexsive space.
\section{ Group  Measure Algebras}

 Let $G$ be a locally compact group with left Haar measure $\lambda$.  A Borel measurable function $\omega\geq1$ on  $G$ is called a weight or weight function if $\omega(xy)\leq\omega(x)\omega(y)$ for all $x,y\in G$.
  Throughout  this section, let $1\leq  p<\infty$ be a real number and  $q$ is such that $1/p+1/q=1$, then   $q$ is  called the exponential conjugate of $p$.  The functions $f:G\rightarrow \mathbb{C}$ such that   $f\omega\in L^p(G)$  form a linear space which is  denoted by  $L^p(G,\omega)$.   Then   $||f||_{p,\omega}=||f\omega||_p$ defines a norm on  $L^p(G,\omega)$. The dual space of $L^1(G,\omega)$ denoted by $L^\infty(G,1/\omega)$.  It consists of all complex-valued measurable functions  $g$ on  $G$ such that  $g/\omega\in L^\infty(G)$.  We equipped $L^\infty(G,1/\omega)$ with  the norm $||g||_{\infty, \omega}=||g/\omega||_\infty$. Then   $$C_0(G,1/\omega)=\{f:G\rightarrow \mathbb{C}:f/\omega\in C_0(G)\}$$  is a subspace of it.
  The dual space of     $L^p(G,\omega) $ is   $ L^q(G,1/\omega)$ consist of all measurable functions $g$ on $G$ such that   $g/\omega\in L^q(G)$ by duality
$$
\langle f,g\rangle:=\int_Gf(x)g(x)d\lambda (x)
$$
for all $f\in L^p(G,\omega)$ and  $g\in L^q(G,1/\omega)$.

For measurable functions $f$ and $g$ on $G$, the
convolution multiplication
$$
(f*g)(x)=\int_G f(y)\;g(y^{-1}x)\;d\lambda(y) \qquad (x\in G),
$$
is defined at each point $x\in G$ for which this makes sense;
i.e., the function  \linebreak $y\mapsto f(y)\;g(y^{-1}x)$ is
$\lambda$-integrable. Then $f*g$ is said to exist if $(f*g)(x)$
exists for almost all $x\in G$.

Since  $\omega\geq1$  hence $L^p(G,\omega)\subseteq L^p(G)$ and  $L^q(G)\subseteq L^q(G,1/\omega)$.

   The map
\[\gamma_p:L^1(G,\omega)\longrightarrow B(L^p(G,\omega))\]
is defined by $\gamma_p(a)(\xi)=a*\xi$ for all  $a\in L^1(G,\omega)$ and $\xi\in L^p(G,\omega)$.

Let $(\gamma_p)_*:L^p(G,\omega)\hat{\otimes} L^{q}(G,1/\omega)\longrightarrow L^\infty(G,1/\omega)$
 be such that
\begin{eqnarray*}
\langle (\gamma_p)_*(\xi\otimes \eta),a\rangle&=&\langle (\xi\otimes \eta),\gamma_p(a)\rangle=\langle \eta,\gamma_p(a)(\xi)\rangle\\
&=&\langle \eta,a*\xi\rangle=\int_G\int_G\eta(t)a(s)\xi(s^{-1}t)d\lambda( t)d\lambda (s)\\
&=&\langle \eta*\check{\xi},a\rangle
\end{eqnarray*}

such that  $ \eta*\check{\xi}(s)=\int_G\xi(s^{-1}t)\eta(t)d\lambda (t)$   and  $\check{\xi}(s)=\xi(s^{-1})$ for all  $s\in G$.
Then   $((\gamma_p)_*)^*=\gamma_p$.

\begin{lemma}
Let  $G$ be a locally compact group, and let $\omega$ be a weight on  $G$. Let   $1\leq  p,q<\infty $.
\begin{enumerate}
\item[(1)]
Every compactly supported function in $L^p(G)$  (respectively  $L^q(G)$) belongs to  $L^p(G,\omega)$ (respectively $L^q(G,1/\omega)$.
\item[(2)]
 $C_{00}(G)$ is dense  in $L^p(G,\omega)$ (respectively $L^q(G,1/\omega)$).
\end{enumerate}
\end{lemma}
\begin{proof}
(1)
By \cite[Lemma 1.3.3]{Kaniuth}, the weight  $\omega $  is bounded away from  both zero and infinity on   compact subsets of $G$. If $f\in L^p(G)$  with compact support $K=supp f$   then for all  $x\in K$, $\omega(x)<b$ for some $b>0$. Then   $||f||_{p,\omega}\leq b ||f||_p$. Hence $f\in L^p(G,\omega)$.

(2)
 By (1) $C_{00}(G)\subseteq L^p(G,\omega)$. To show that  $C_{00}(G)$  is dense in $L^p(G,\omega)$, let
  $f\in L^p(G,\omega)$ and $\epsilon>0$ be given. Since  $f\omega\in L^p(G)$, there exists  $h\in C_{00}(G)$ such that  $||h-f\omega||_p^p\leq \epsilon$.  Let $S$ denote the compact support of $h$ and observe that  $\omega(x)\geq\delta$ for some  $\delta>0$ and all  $x\in S$.  Since $\omega$ is bounded  on $S$, $\omega_{|_S}\in L^p(S)$ and hence there exist a continuous function $\eta:S\rightarrow \mathbb{R}$ such that $\eta(x)\geq \delta$ for all $x\in S$ and
\[\int_S|\eta(x)-\omega(x)|^pd\lambda(x)\leq \frac{\epsilon\delta ^p}{||h||_\infty^p}.\]
Now  define a function $g$ on $G$ by $g(x)=\frac{h(x)}{\eta(x)}$ for  $x\in S$ and   $g(x)=0$ for $x\not\in S$.  Since $ 1/\eta(x)\leq 1/\delta $  for all $x\in S$, it is easily verified that  $g$ is continuous on $G$. Thus $g\in C_{00}(G)$ and
\[||g-f||_{p,\omega}^p=\int_S\omega(x)^p|g(x)-f(x)|^pd\lambda(x)+\int_{G\setminus S}\omega(x)^p|f(x)|^pd\lambda(x),\]
We estimate the first integral on the right as follows:
\begin{eqnarray*}
\int_S \omega(x)^p|g(x)-f(x)|^pd\lambda(x)&\leq&\int_S\omega(x)^p|\frac{h(x)}{\eta(x)}-\frac{h(x)}{\omega(x)}|^p d\lambda(x)\\
&+&\int_S\omega(x)^p|\frac{h(x)}{\omega(x)}-f(x)|^pd\lambda(x)\\
&=&\int_S\frac{h(x)^p}{\eta(x)^p}|\omega(x)-\eta(x)|^pd\lambda(x)\\
&+&\int_S|h(x)-\omega(x)f(x)|^pd\lambda(x)\\
&\leq&\frac{||h||_\infty^p}{\delta^p}\int_S|\omega(x)-\eta(x)|^pd\lambda(x)\\
&+&\int_S|h(x)-\omega(x)f(x)|^pd\lambda(x)\\
&\leq&\epsilon+\int_S|h(x)-\omega(x)f(x)|^pd\lambda(x).
\end{eqnarray*}
It follows that
\begin{eqnarray*}
||g-f||_{p,\omega}^p&\leq&\epsilon+\int_S|h(x)-\omega(x)f(x)|^pd\lambda(x)+\int_{G\setminus S}\omega(x)^p|f(x)|^pd\lambda(x)\\
&=&\epsilon +\int_G|h(x)-\omega(x)f(x)|^pd\lambda(x)\leq 2\epsilon.
\end{eqnarray*}
This shows that $C_{00}(G)$ is dense in $L^p(G,\omega)$.

For $q$ is similar to $p$, but $g$ must be defined by  $g(x)=\eta(x)h(x)$.
\end{proof}
\begin{remark}
If $\alpha\leq\omega\leq \beta$, then  $L^p(G,\omega)=L^p(G)$ and the norms        $||.||_{p,\omega}$ and $||.||_p$ are equivalent and $C_{00}(G))$ is dense in  $L^p(G)$, hence dense in  $L^p(G,\omega)$.
\end{remark}
For   the function $\xi$ on $G$  and  $x\in G$ left and right translates $L_x\xi$ and $R_x\xi$ defined by $L_x\xi(y)=\xi(x^{-1}y)$ and $R_x\xi(y)=\xi(yx)$  for all  $y\in G$.
\begin{lemma}\label{approximate}
Let $G$  be a locally compact group, and let $\omega$ be a weight on  $G$,  $1< p<\infty $  and $\xi\in L^p(G,\omega)  $.
\begin{enumerate}
\item[(1)]for all $x\in G$, $L_x\xi\in L^p(G,\omega)$  and $||L_x\xi||_{p,\omega}\leq \omega(x)||\xi||_{p,\omega}$.
\item[(2)]
The map $x\rightarrow L_x$ from  $G$ into $L^p(G,\omega)$ is continuous.

\item[(3)] Let  $q$ be the exponential conjugate of $p$. If    $\omega$ is continuous on  $G$  and
  $\eta\in  L^{q}(G,1/\omega)$.   Then  $ \eta*\check{\xi}\in C_0(G,1/\omega)$.
\item[(4)]
For every relatively compact neighborhood $V$ of $e$ in $G$, let $u_V\in L^p(G,\omega)$ be such that  $u_V\geq0$ and  $||u_V||_{p,\omega}=1$ and  $u_V=0$ almost every where on $G\setminus V$. Then, given  $\xi\in L^p(G,\omega)$ and $\epsilon>0$,

  \[||u_V*\xi-\xi||_{p,\omega}<\epsilon\]
for all sufficiently small  $V$.
\end{enumerate}
\end{lemma}
\begin{proof}
(1)  follows simply from  submultiplicativity of    $\omega$ :
 \begin{eqnarray*}
||L_x\xi||_{p,\omega}^p&=&\int_G|\xi(x^{-1}t)|^p\omega(t)^pd\lambda(t)\\
&=&\int_G|\xi(x^{-1}t)|^p{\omega(x^{-1}t)^p}\frac{\omega(t)^p}{\omega(x^{-1}t)^p}d\lambda(t)\\
&\leq&\omega(x)^p\int_G|\xi(x^{-1}t)|^p\omega(x^{-1}t)^pd\lambda(t)\\
&=&\omega(x)^p||\xi||_{p,\omega}^p
\end{eqnarray*}
(2)Let  $\epsilon>0$.
Since  $C_{00}(G)$ is dense in  $L^p(G,\omega)$.  There exists $g\in C_{00}(G)$ such that
$||\xi-g||_{p,\omega}<\epsilon/3$.
Let  $x\in G$ and choose a compact neighbourhood  $V$ of $x$ in $G$. Let
\[C=\sup\{\omega(s):s\in V.supp g\}<\infty\]
 Then for $y\in V$,

 \begin{eqnarray*}
||L_xg-L_yg||_{p,\omega}^p&=&\int_G|g(x^{-1}t)-g(y^{-1}t)|^p\omega(t)^pd\lambda(t)\\
&\leq&C^p\int_{V.supp g}|g(x^{-1}t)-g(y^{-1}t)|^pd\lambda(t)\\
&=&C^p||L_xg-L_yg||_{p}^p
\end{eqnarray*}
 when $y\rightarrow x$ converges to zero.
Since  $\omega$ is  locally bounded :
 \begin{eqnarray*}
||L_x\xi-L_y\xi||_{p,\omega}&\leq & ||L_x(\xi-g)||_{p,\omega}+||L_xg-L_yg||_{p,\omega}+||L_y(\xi-g)||_{p,\omega}\\
&\leq&(\omega(x)+\omega(y))||\xi-g||_{p,\omega}+||L_xg-L_yg||_{p,\omega}<\epsilon
\end{eqnarray*}
 (3)
For $x\in G$, by H$\ddot{o}$lder inequality:
 \[\int_G|\xi(x^{-1}y)\eta(y)|d\lambda(y)=\int_G|L_{x}\xi(y)|.|\eta(y)|d\lambda(y)\leq||L_{x}\xi||_{p,\omega}||\eta||_{q,\omega}.\]
   So  $ \eta*\check{\xi} $ is defined everywhere and bounded on  $G$  by $\omega(x)||\xi||_{p,\omega}||\eta||_{q,\omega}$.  For $x,y\in G$, H$\ddot{o}$lder inequality gives
  \[| \eta*\check{\xi}(x)- \eta*\check{\xi}(y)|\leq ||L_{x}\xi-L_{y}\xi||_{p,\omega}||\eta||_{q,\omega}.\]
 The map $t\rightarrow L_t\xi$ from $G$ into  $L^p(G,\omega)$ is continuous and therefore we obtain that  $ \eta*\check{\xi} $ is continuous.
To prove $ \eta*\check{\xi}\in C_0(G,1/\omega)$, note first that $\xi,\eta\in C_{00}(G)$ whenever  $\eta*\check{\xi}\in C_{00}(G)$.

If  $\xi\in L^p(G,\omega) , \eta\in L^q(G,1/\omega)$, then for $1\leq r<\infty$, since  $C_{00}(G)$ is dense in  $L^r(G,\omega)$ and  $L^r(G,1/\omega)$, there exist  $(\xi_n)$ and $(\eta_n)$ in $C_{00}(G)$ such that $||\xi-\xi_n||_{p,\omega}\rightarrow 0$ and
$||\eta-\eta_n||_{q,\omega}\rightarrow 0$.  Then for all  $x\in G$,
\begin{eqnarray*}
|\eta*\check{\xi}(x)-\eta_n*\check{\xi}_n(x)|/{\omega(x)}&\leq&|\eta*(\check{\xi}-\check{\xi}_n)(x)|/{\omega(x)}+|(\eta-\eta_n)*\check{\xi}(x)|/{\omega(x)}\\
&\leq&||\xi-\xi_n||_{p,\omega}||\eta||_{q,\omega}+||\xi||_{p,\omega}||\eta-\eta_n||_{q,\omega},
\end{eqnarray*}
which tends  to $0$ as $n\rightarrow \infty$. It follows that $(\eta*\check{\xi})/\omega\in C_0(G)$.

(4)
Since  $C_{00}(G)$ is dense in $L^p(G,\omega)$, we can choose  $g\in C_{00}(G)$ such that
 $||\xi-g||_{p,\omega}<\frac{\epsilon}{3} $.  For  $g$ it follows that
\begin{eqnarray*}
||u_V*g-g||_{p,\omega}^p&=&\int_G|\int_Gu_V(xy)g(y^{-1})d\lambda(y)-g(x)|^p\omega(x)d\lambda(x)\\
&=&\int_G|\int_Gu_V(y)L_yg(x)d\lambda(y)-g(x)|^p\omega(x)d\lambda(x)\\
&=&\int_G|\int_Gu_V(y)[L_yg(x)-g(x)]d\lambda(y)|^p\omega(x)d\lambda(x)\\
&\leq&\int_G\left (\int_Gu_V(y)|L_yg(x)-g(x)|d\lambda(y)\right )^p\omega(x)d\lambda(x)\\
&\leq&\lambda(V.supp g).\sup \{||L_yg-g||_\infty^p:y\in V\}.
\end{eqnarray*}
 Now, since the map $y\rightarrow L_yg$ from  $G$ into  $L^p(G,\omega)$ is continuous, we find a neighbourhood $W$ of $e$ in  $G$ such that, for all $y\in W$,
\[||L_yg-g||_\infty\leq \frac{\epsilon}{3\lambda(V.supp g)^{1/p}}\]

Together with the above estimate we get for all $V\subseteq W$,
\begin{eqnarray*}
||u_V*\xi-\xi||_{p,\omega}&\leq&||u_V*(\xi-g)||_{p,\omega}+||u_V*g-g||_{p,\omega}+||g-\xi||_{p,\omega}\\
&\leq& (||u_V||_{p,\omega}+1)||\xi-g||_{p,\omega}+\frac{\epsilon}{3}<\epsilon
\end{eqnarray*}

\end{proof}

Now
\[\tilde {\Theta}: M_b(G,\omega)\longrightarrow B(L^p(G,\omega))\]
 extends  $\lambda_p$ to $M_b(G,\omega)$ by
 \[\langle\eta,\tilde{\Theta}(\mu)(\xi)\rangle=\langle\mu,\eta*\check{\xi}\rangle.\]

Then for $\xi,\eta\in C_{00}(G	 )$, we have
\[\langle\eta,\tilde{\Theta}(\mu)(\xi)\rangle=\langle\eta, \mu*\xi\rangle\]
 where $\mu*\xi(t)=\int_G\xi(s^{-1}t)d\mu(s)$.

The next theorem extends  the Young's construction \cite[Theorem4]{y1} for weighted convolution  measure algebras.
\begin{theorem}
 Let $G$  be a locally compact group, and let  $\omega\geq1$ be a continuous  weight on $G$.
Then weighted convolution measure  algebra  $M_b(G,\omega)$ is a $\mathcal{WAP}$-algebra.
\end{theorem}
\begin{proof}

Let $\{p_n\}$ be some sequence in $(1,\infty)$ such that  $p_n\longrightarrow1$. Let $E$ be the direct sum, in an  $\ell_2$- sense, of the spaces  $ L^{p_n}(G,\omega) $.  To be exact,
 \[E=\{\{\xi_n\}:\xi_n\in L^{p_n}(G,\omega), ||\{\xi_n\}||_E<\infty\}\]
  with $||\{\xi_n\}||_E=(\sum_{n=1}^\infty ||\xi_n||^2_{p_n,\omega})^{\frac{1}{2}}$, then  $E$  is reflexive.

The mapping  $\Theta:M_b(G,\omega)\longrightarrow B(E)$  is defined by  $\Theta(\mu)(\{\xi_n\})=\{\mu*\xi_n\}$.
Consider the adjoint map $\Theta_*:E\hat{\otimes}E^*\longrightarrow M_b(G,\omega)^*$  given by
\begin{eqnarray*}
\langle\Theta_*(\xi\otimes \eta),\mu\rangle&=&\langle\eta,\Theta(\mu)(\xi)\rangle\\
&=&\sum_{n=1}^\infty\langle\eta_n,\tilde{\Theta}(\mu)(\xi_n)\rangle\\
&=&\sum_{n=1}^\infty\langle\mu,\eta_n*\check{\xi}_n\rangle
\end{eqnarray*}
 where $\xi=\{\xi_n\}\in E, \eta=\{\eta_n\}\in E^*$ and $\mu\in M_b(G,\omega)$.

  In particular, $\Theta_*$ maps into $C_0(G,1/\omega)$. So that $\Theta$ is  $\omega^*$- $\omega^*$ continuous.

 For  $\xi,\eta\in  C_{00}(G)$, we have that
 $$\lim_{p\longrightarrow 1}||\xi||_{p,\omega}=||\xi||_{1,\omega}\quad , \quad \lim_{q\longrightarrow\infty }||\eta||_{q,\omega}=||\eta||_{\infty,\omega}.$$
  By lemma \ref{approximate}  for any $\eta\in C_{00}(G)$ and  $\varepsilon>0$ we can find some    $u_V\in C_{00}(G)$ with   $||u_V||_{1,\omega}=1$  and  $|| \eta*\check{u}_V-\eta||_{\infty,\omega}<\varepsilon$.
As  $p_n\longrightarrow 1$, we can find $n\in \mathbb{N}$ with $||u_V||_{p_n,\omega}<1+\varepsilon$
 and $||\eta||_{q_n,\omega}<(1+\varepsilon)||\eta||_{\infty,\omega}$.
It follows that
 \[||\eta*\check{u}_V||_{\infty,\omega}\leq ||u_V||_{p_n,\omega}.||\eta||_{q_n,\omega}<(1+\varepsilon)^2||\eta||_{\infty,\omega}\]
and that  \[|\langle \eta,\Theta(\mu)(u_V)\rangle|=|\langle\mu,\eta*\check{u}_V\rangle|\geq|\langle\mu,\eta\rangle| -\varepsilon||\mu||\]
By taking suitable supremums, it now follows:
\begin{eqnarray*}
(1-\epsilon)||\mu||&=&\sup\{|\langle \mu,\eta\rangle|: \eta\in C_{00}(G), ||\eta||\leq1\}-\epsilon||\mu||\\
&\leq &\sup\{|\langle \eta,\Theta(\mu)(u_V)\rangle|:\eta\in C_{00}(G), ||\eta||\leq1\}\\
&=&||\Theta(\mu)(u_V)||\leq ||\Theta||||\mu||
\end{eqnarray*}
 Since $\epsilon$ is arbitrary  $\Theta$ is isometric.
 \end{proof}

Let $G$ be a locally compact group, and let $\omega$ be a Borel-measurable weight function on it. Then the Fourier-Stieltjes algebra $B(G)$ and the  weighted measure algebra $M_b(G,\omega)$ are  dual Banach algebras(see \cite{Bkh}), also the Fourier algebra $A(G)$ and $L^1(G,\omega)$ are   closed ideals  in them respectively. Hence all are $\mathcal{WAP}$-algebras.

The next example shows that  Young's construction  can not work for semigroups.

\begin{examples}
Let $S=(\mathbb{N}, \min)$. Then  $\ell_1(S)$ is $\mathcal{WAP}$-algebras. But we can't apply Young's construction for this semigroup. Let  $f(n)=\frac{1}{n^2}$  and $g(n)=\frac{1}{\sqrt[3]{n}}$ defined for all $n\in \mathbb{N}$.

\[||f||_1=\sum_{n\in\mathbb{N}}\frac{1}{n^2}<\infty,||g||_4=\sum_{n\in \mathbb{N}}\frac{1}{n^{\frac{4}{3}}}<\infty\]
Then  $f\in \ell_1(S)$ and  $g\in \ell_4(S)$ but  $f*g\not\in \ell_4(S)$ since
\[f*g(k)=\sum _{n.m=k}f(n)g(m)=\sum _{m=k}\frac{1}{k^2} \frac{1}{\sqrt[3]{m}}+\sum _{n=k}\frac{1}{n^2} \frac{1}{\sqrt[3]{k}}=\infty\]
for all  $k\in \mathbb{N}$.

\end{examples}

\section{Relation between $C_0(S)$ and $\mathcal{WAP}$-algebra}

  Let $S=\mathbb{N}$. Then for $S$ equipped with $\min$ multiplication,  the semigroup algebra $\ell_1(S)$  is a WAP-algebra  but is not neither  Arens regular nor a dual Banach algebra. While, if we replace the $\min$ multiplication  with $\max$ then $\ell_1(S)$ is a dual Banach algebra (so a WAP-algebra) which is not Arens regular. If we change the multiplication of $S$ to the zero multiplication then the resulted semigroup algebra is Arens regular (so a WAP-algebra) which is not a dual Banach algebra. This describes the  interrelation  between the concepts of  being Arens regular  algebra, dual Banach algebra and  WAP-algebra.

\begin{definition}{\rm Let $X$, $Y$ be sets and $f$ be a complex-valued function on $X\times Y$.
\begin{enumerate}
  \item We say that $f$ is a cluster on $X\times Y$ if for each pair of sequences $(x_n)$, $(y_m)$ of distinct elements of $X,Y$, respectively
\begin{equation}\label{fg}
  \lim_n\lim_mf(x_n,y_m)=\lim_m\lim_nf(x_n,y_m)
\end{equation}
whenever both sides of (\ref{fg}) exist.
  \item If $f$ is cluster and both sides of \ref{fg} are zero (respectively positive) in all cases, we say that $f$ is $0$-cluster(respectively positive cluster).
\end{enumerate}

}\end{definition}
In general $\{f\omega:f\in wap(S)\}\not= wap(S,1/\omega)$.
By using \cite[Lemma1.4]{BR} the following is immediate.
\begin{lemma}\label{ghjk}{\rm
Let $\Omega(x,y)=\frac{\omega(xy)}{\omega(x)\omega(y)}$, for $x,y\in S$. Then
\begin{enumerate}
  \item If $\Omega$ is cluster, then $\{f\omega:f\in wap(S)\}\subseteq wap(S,1/\omega)$;
  \item If $\Omega$ is positive cluster, then $wap(S,1/\omega)=\{f\omega:f\in wap(S)\}$.

\end{enumerate}
}\end{lemma}
It should be noted that if $M_b(S)$ is Arens regular (resp. dual Banach algebra) then $M_b(S,\omega)$ is so.  We don't know that if $M_b(S)$ is {\rm WAP}-algebra, then $M_b(S,\omega)$ is so. The following Lemma give a partial answer to this question.
\begin{corollary}\label{weighted} {\rm Let $S$ be  a locally compact  topological semigroup with a Borel measurable weight function   $\omega$ such that $\Omega$ is cluster on $S\times S$.
 \begin{enumerate}
 \item  If $M_b(S)$ is a {\rm WAP}-algebra, then so is $M_b(S,\omega)$;
 \item If $\ell_1(S)$ is a {\rm WAP}-algebra, then  so is $\ell_1(S,\omega)$.
 \end{enumerate}
}\end{corollary}
\begin{proof}(1) Suppose that $M_b(S)$ is a {\rm WAP}-algebra so $wap(S)$ separates the points of $S$. By lemma\ref{ghjk} for every $f\in wap(S)$, $f\omega\in wap(S,1/\omega)$. Thus the evaluation map $\epsilon:S\longrightarrow \tilde {X}$ is   one to one.

  (2) follows from (1).
\end{proof}
\begin{corollary}\label{homomo} {\rm For a locally compact  semi-topological semigroup $S$,
  \begin{enumerate}
 \item  If $C_0(S)\subseteq wap(S)$, then  the measure algebra $M_b(S)$ is a {\rm WAP}-algebra.
\item  If  $S$ is discrete and $c_0(S)\subseteq wap(S)$, then $\ell_1(S)$ is a {\rm WAP}-algebra.
 \end{enumerate}
}\end{corollary}
\begin{proof}(1)
By \cite[Corollary 4.2.13]{BJM} the map $\epsilon:S\longrightarrow S^{wap}$ is one to one, thus $M_b(S)$ is a {\rm WAP}-algebra.

(2) follows from (1).
\end{proof}

Dales, Lau and Strauss \cite[Theorem 4.6, Proposition 8.3]{DLS} showed that for a semigroup $S$, $\ell^1(S)$ is a dual Banach algebra with respect to $c_0(S)$ if and only if $S$ is weakly cancellative. If $S$ is  left or right weakly cancellative semigroup, then $\ell^1(S)$ is a {\rm WAP}-algebra. The next example shows that the converse is not true, in general.
\begin{example}
Let $S=(\mathbb{N},\min)$ then $wap(S)=c_0(S) \oplus \mathbb{C} $. So $\ell^1(S)$ is a {\rm WAP}-algebra but $S$ is neither left nor right weakly cancellative.
In fact, for $f\in wap(S)$  and all sequences $\{a_n\}$, $\{b_m\}$ with distinct element in $S$, we have $\lim_mf(b_m)=\lim_m\lim_nf(a_nb_m)=\lambda=\lim_n\lim_mf(a_nb_m)=\lim_nf(a_n)$, for some $\lambda\in \mathbb{C}$. This means $f-\lambda\in c_0(S)$ and $wap(S)\subseteq c_0(S) \oplus \mathbb{C} $. The other inclusion is clear.
\end{example}


 If $\{x_n\}$ and $\{y_m\}$ are sequences in $S$ we obtain an infinite matrix $\{x_ny_m\}$ which has $x_ny_m$ as its entry in the $m$th row and $n$th column.
 As in \cite{BR}, a matrix is said to be of row type $C$ ( resp. column type $C$) if the rows  ( resp. columns ) of the matrix are all constant and distinct. A matrix is of type $C$ if it is constant or of row or column type $C$.

J.W.Baker and A. Rejali in  \cite[Theorem 2.7(v)]{BR} showed that $\ell^1(S)$ is Arens regular if and only if for each pair of sequences $\{x_n\}$, $\{y_m\}$  with distinct elements in $S$ there is a submatrix  of $\{x_ny_m\}$ of type $C$.

A matrix $\{x_ny_m\}$  is said to be upper triangular constant if $x_ny_m=s$ if and only if $m\geq n$ and it is lower triangular constant if  $x_ny_m=s$ if and only if $m\leq n$. A matrix $\{x_ny_m\}$  is said to be  $W$-type if every submtrix of $\{x_ny_m\}$ is  neither upper triangular constant nor lower triangular constant.

\begin{theorem}{\rm
Let $S$ be a semigroup. The following statements are equivalent:
\begin{enumerate}
  \item $c_0(S)\subseteq wap(S)$.
  \item For each $s\in S$ and each pair $\{x_n\}$, $\{y_m\}$ of sequences in $S$,
  $$\{\chi_s(x_ny_m):n<m\}\cap \{\chi_s(x_ny_m):n>m\}\not=\emptyset;$$
\item For  each pair $\{x_n\}$, $\{y_m\}$ of sequences in $S$ with distinct elements, $\{x_ny_m\}$ is a  $W$-type matrix;

\item For every $s\in S$, every infinite set
    $B\subset S$ contains a finite subset $F$ such that $\cap\{sb^{-1}:b\in
    F\}\setminus (\cap\{sb^{-1}: b\in B\setminus F\})$ and $\cap\{b^{-1}s:b\in F\}\setminus (\cap\{b^{-1}s: b\in B\setminus F\})$ are finite.
\end{enumerate}
}\end{theorem}
\begin{proof}(1)$\Leftrightarrow$ (2). For all $s\in S$,  $\chi_s\in wap(S)$ if and only if $$\{\chi_s(x_ny_m):n<m\}\cap \{\chi_s(x_ny_m):n>m\}\not=\emptyset.$$

 (3)$\Rightarrow$ (1)Let $c_0(S)\not\subseteq wap(S)$ then there are sequences $\{x_n\}$, $\{y_m\}$ in $S$ with distinct elements such that for some $s\in S$, $$1=\lim_m\lim_n\chi_s(x_ny_m)\not=\lim_n\lim_m\chi_s(x_ny_m)=0.$$


 Since $\lim_n\lim_m\chi_s(x_ny_m)=0$,  for $1>\varepsilon>0$ there is a $N\in\mathbb{N}$ such that for all $n\geq N$, $\lim_m\chi_s(x_my_n)<\varepsilon$. This implies for all $n\geq N$, $\lim_m\chi_s(x_my_n)=0$. Then for $n\geq N$, $1>\varepsilon>0$ there is a $M_n\in\mathbb{N}$ such that for all $m\geq M_n$ we have $\chi_s(x_my_n)<\varepsilon$. So if we omit finitely many terms, for all $n\in \mathbb{N}$ there is $M_n\in\mathbb{N}$ such that for all $m\geq M_n$ we have $x_my_n\not=s$. As a similar argument, for all $m\in \mathbb{N}$ there is $N_m\in\mathbb{N}$ such that for all $n\geq N_m$, $x_my_n=s$.

  Let $a_1=x_1$, $b_1$ be the first $y_n$ such that $a_1y_n=s$. Suppose $a_m, b_n$ have been chosen for $1\leq m,n<r$, so that $a_nb_m=s$ if and only if $n\geq m$. Pick $a_r$  to be the first $x_m$ not belonging to the finite set $\cup_{1\leq n\leq r}\{x_m:x_my_n=s\}$. Then $a_rb_n\not=s$ for $n<r$. Pick $b_r$ to be the first $y_n$ belonging to the cofinite set  $\cap_{1\leq n\leq r}\{y_n:x_my_n=s\}$. Then $a_nb_m=s$ if and only if $n\geq m$. The sequences $(a_m)$, $(b_n)$ so constructed satisfy $a_mb_n=s$ if and only if $n\geq m$.
 That is, $\{a_nb_m\}$ is not of $W$-type  and this is a contradiction.

(1)$\Rightarrow$ (3). Let there are sequences $\{x_n\}$, $\{y_m\}$ in $S$ such that $\{x_ny_m\}$ is not a  $W$-type matrix, (say) $x_ny_m=s$ if and only if $m\leq n$. Then $$1=\lim_m\lim_n\chi_s(x_ny_m)\not=\lim_n\lim_m\chi_s(x_ny_m)=0.$$
 So $\chi_s\not\in wap(S)$. Thus $c_0(S)\not\subseteq wap(S)$.

(4)$\Leftrightarrow$ (1)This is Ruppert criterion for $\chi_s\in wap(S)$,  see \cite[Theorem 4]{Ruppert}.

\end{proof}
\begin{example}
\begin{enumerate}

  \item[(i)]
    Let $S$ be the interval $[\frac{1}{2},1]$ with multiplication $x.y=\max\{\frac{1}{2},xy\}$, where $xy$ is the ordinary multiplication on $\mathbb{R}$. Then for all $s\in S\setminus\{\frac{1}{2}\}$, $x\in S$, $x^{-1}s$ is finite. But $x^{-1}\frac{1}{2}=[\frac{1}{2},\frac{1}{2x}]$. Let $B=[\frac{1}{2},\frac{3}{4})$. Then for all finite subset $F$ of $B$,
    $$\bigcap_{x\in F}x^{-1}\frac{1}{2}\setminus \bigcap_{x\in B\setminus F}x^{-1}\frac{1}{2}=[\frac{2}{3},\frac{1}{2x_F}]$$
 where $x_F=\max F$. By \cite[Theorem 4]{Ruppert}   $\chi_{\frac{1}{2}}\not\in wap(S)$. So $c_0(S\setminus\{\frac{1}{2}\})\oplus \mathbb{C}\subsetneqq wap(S)$. It can be readily verified that $\epsilon:S\longrightarrow S^{wap}$ is   one to one, so $\ell_1(S)$ is a WAP-algebra but $c_0(S)\not\subseteq wap(S)$. This is a counter example for the converse of Corollary \ref{homomo}.

 \item[(ii)] Take $T=(\mathbb{N}\cup\{0\},.)$ with 0  as zero
 of $T$ and the multiplication defined by

$$n.m=\left\{
\begin{array}{lr}
                                                     n & \mbox{if}\quad n=m \\
                                                       0 &
                                                       \mbox{otherwise}.
                                                     \end{array}\right.
$$

Then $S=T\times T$ is a semigroup with coordinate wise multiplication.  Now let
$X=\{(k,0): k\in T\}$, $Y=\{(0,k):k\in T\}$ and $Z=X\cup Y$. We
use the Ruppert criterion \cite{Ruppert} to show that
$\chi_z\not\in wap(S)$, for each $z\in Z$. Let $B=\{(k,n): k,n\in
T\}$, then $(k,n)^{-1}(k,0)=\{(k,m):m\not= n\}=B\setminus
\{(k,n)\}$. Thus for all finite subsets $F$ of $B$,
\begin{eqnarray*}
\left(\cap\{(k,n)^{-1}(k,0): (k,n)\in F\}\right)&\setminus&(
\cap\{(k,n)^{-1}(k,0):
(k,n)\in B\setminus F\})\\
&=& \left(\cap\{(k,0)(k,n)^{-1}: (k,n)\in F\}\right ) \\
&\setminus &(\cap\{(k,n)^{-1}(k,0): (k,n)\in B\setminus F\}) \\
&=& (B\setminus F) \setminus F=B\setminus F
\end{eqnarray*}
and the last set is infinite. This means $\chi_{(k,0)}\not\in wap(S)$.  Similarly $\chi_{(0,k)}\not\in wap(S)$.  Let $f=\sum_{n=0}f(0,n)\chi_{(0,n)}+\sum_{m=1}^\infty f(m,0)\chi_{(m,0)}$ be in $ wap(S)$. For arbitrary fixed $n$ and sequence $\{(n,k)\}$ in $S$, we have $\lim_k f(n,k)=\lim_k\lim_lf(n,l.k)=\lim_l\lim_kf(n,l.k)=f(n,0)$ implies $f(n,0)=0$. Similarly $f(0,n)=0$ and $f(0,0)=0$. Thus $f=0$.   In fact $wap(S)\subseteq \ell^\infty(\mathbb{N}\times\mathbb{N})$. Since $wap(S)$ can not separate the points of $S$ so $\ell_1(S)$ is not a WAP-algebra.
  Let $\omega(n,m)=2^n3^m$ for $(n,m)\in S$. Then $\omega$ is a weight on $S$  such that $\omega\in wap(S,1/\omega)$, so the evaluation map $\epsilon:S\longrightarrow \tilde {X}$ is   one to one. This means $\ell_1(S,\omega)$ is a WAP-algebra but $\ell_1(S)$ is not a WAP-algebra. This is a counter example for the converse of Corollary \ref{weighted}.
  \item[(iii)]
    Let $S$ be the interval $[\frac{1}{2},1]$ with multiplication $x.y=\max\{\frac{1}{2},xy\}$, where $xy$ is the ordinary multiplication on $\mathbb{R}$. Then for all $s\in S\setminus\{\frac{1}{2}\}$, $x\in S$, $x^{-1}s$ is finite. But $x^{-1}\frac{1}{2}=[\frac{1}{2},\frac{1}{2x}]$. Let $B=[\frac{1}{2},\frac{3}{4})$. Then for all finite subset $F$ of $B$,
    $$\bigcap_{x\in F}x^{-1}\frac{1}{2}\setminus \bigcap_{x\in B\setminus F}x^{-1}\frac{1}{2}=[\frac{2}{3},\frac{1}{2x_F}]$$
 where $x_F=\max F$. By \cite[Theorem 4]{Ruppert}   $\chi_{\frac{1}{2}}\not\in wap(S)$. So $c_0(S\setminus\{\frac{1}{2}\})\oplus \mathbb{C}\subsetneqq wap(S)$. It can be readily verified that $\epsilon:S\longrightarrow S^{wap}$ is   one to one, so $\ell_1(S)$ is a WAP-algebra but $c_0(S)\not\subseteq wap(S)$. This is a counter example for the converse of Corollary \ref{homomo}.

 \item[(iv)] Take $T=(\mathbb{N}\cup\{0\},.)$ with 0  as zero
 of $T$ and the multiplication defined by

$$n.m=\left\{
\begin{array}{lr}
                                                     n & \mbox{if}\quad n=m \\
                                                       0 &
                                                       \mbox{otherwise}.
                                                     \end{array}\right.
$$

Then $S=T\times T$ is a semigroup with coordinate wise multiplication.  Now let
$X=\{(k,0): k\in T\}$, $Y=\{(0,k):k\in T\}$ and $Z=X\cup Y$. We
use the Ruppert criterion \cite{Ruppert} to show that
$\chi_z\not\in wap(S)$, for each $z\in Z$. Let $B=\{(k,n): k,n\in
T\}$, then $(k,n)^{-1}(k,0)=\{(k,m):m\not= n\}=B\setminus
\{(k,n)\}$. Thus for all finite subsets $F$ of $B$,
\begin{eqnarray*}
\left(\cap\{(k,n)^{-1}(k,0): (k,n)\in F\}\right)&\setminus&(
\cap\{(k,n)^{-1}(k,0):
(k,n)\in B\setminus F\})\\
&=& \left(\cap\{(k,0)(k,n)^{-1}: (k,n)\in F\}\right ) \\
&\setminus &(\cap\{(k,n)^{-1}(k,0): (k,n)\in B\setminus F\}) \\
&=& (B\setminus F) \setminus F=B\setminus F
\end{eqnarray*}
and the last set is infinite. This means $\chi_{(k,0)}\not\in wap(S)$.  Similarly $\chi_{(0,k)}\not\in wap(S)$.  Let $f=\sum_{n=0}f(0,n)\chi_{(0,n)}+\sum_{m=1}^\infty f(m,0)\chi_{(m,0)}$ be in $ wap(S)$. For arbitrary fixed $n$ and sequence $\{(n,k)\}$ in $S$, we have $\lim_k f(n,k)=\lim_k\lim_lf(n,l.k)=\lim_l\lim_kf(n,l.k)=f(n,0)$ implies $f(n,0)=0$. Similarly $f(0,n)=0$ and $f(0,0)=0$. Thus $f=0$.   In fact $wap(S)\subseteq \ell^\infty(\mathbb{N}\times\mathbb{N})$. Since $wap(S)$ can not separate the points of $S$ so $\ell_1(S)$ is not a WAP-algebra.
  Let $\omega(n,m)=2^n3^m$ for $(n,m)\in S$. Then $\omega$ is a weight on $S$  such that $\omega\in wap(S,1/\omega)$, so the evaluation map $\epsilon:S\longrightarrow \tilde {X}$ is   one to one. This means $\ell_1(S,\omega)$ is a WAP-algebra but $\ell_1(S)$ is not a WAP-algebra. This is a counter example for the converse of Corollary \ref{weighted}.
  
\end{enumerate}
\end{example}
\noindent{{\bf Acknowledgments.}}  This research was supported by
the Centers of Excellence for Mathematics at the University of
Isfahan.

\end{document}